\documentclass[11pt,english]{amsart}

\usepackage{amsmath}
\usepackage{amssymb}
\usepackage{amsthm}
\usepackage{amscd}
\usepackage{babel}
\usepackage{graphicx}
\usepackage{algpseudocode}
\usepackage{algorithm}
\usepackage{doi}
\usepackage{nicefrac}

\usepackage{enumerate}

\usepackage[margin=1in]{geometry}

\newtheorem{teor}{Theorem}
\newtheorem{cor}{Corollary}

\newtheorem{con}{Conjecture}
\newtheorem{lem}{Lemma}

\theoremstyle{definition}

\newtheorem{defi}{Definition}

\renewcommand{\subjclassname}{AMS \textup{2010} Mathematics Subject Classification\ }
\author{Max A. Alekseyev}
\address{Department of Mathematics, George Washington University\\
Washington, DC, USA} 
\email{maxal@gwu.edu}

\author{Jos\'{e} Mar\'{i}a Grau}
\address{Departamento de Matematicas, Universidad de Oviedo\\ Avda. Calvo Sotelo s/n, 33007 Oviedo, Spain}
\email{grau@uniovi.es}

\author{Antonio M. Oller-Marc\'{e}n}
\address{Centro Universitario de la Defensa de Zaragoza\\ Ctra. Huesca s/n, 50090 Zaragoza, Spain} 
\email{oller@unizar.es}

\DeclareMathOperator{\PrimeParts}{\textsc{PrimeParts}}
\DeclareMathOperator{\AllParts}{\textsc{AllParts}}

\title{Computing solutions to the congruence \\$\boldsymbol{1^n + 2^n + \dotsb + n^n\equiv p \pmod{n}}$}

\begin{document}

\begin{abstract}
It is well-known that the congruence $\sum_{i=1}^{ n} i^{ n} \equiv 1 \pmod{n}$ has exactly five solutions: $\{1,2,6,42,1806\}$. 
In this work, we characterize the solutions to the congruence $1^n + 2^n + \dotsb + n^n\equiv p \pmod{n}$ for every prime $p$. 
This characterization leads to an algorithm for computing all such solutions, when there is a finite number of them. More generally, our algorithm enables computing all the solutions below a much higher bound as compared to what can be achieved by a naive exhaustive search.
\end{abstract}

\maketitle
\subjclassname{11B99, 11A99, 11A07}


\section{Introduction}

There exist many Diophantine equations with ``few'' known solutions, whose search is hard both from the theoretical and computational points of view.
One of the best-known examples is given by the \emph{Erd\"{o}s--Moser equation} $\sum_{i=1}^{m-1}i^n =m^n$, 
for which it has been proved~\cite{BUT} that there is only a trivial solution $1^1+2^1=3^1$ when $m<1.485\cdot 10^{9321155}$. 
Other famous examples include \emph{Giuga's conjecture}~\cite{GIU} stating non-existence of composite numbers $n$ such that $\sum_{i=1}^{n-1}i^n \equiv -1 \pmod{n}$, 
which has been verified~\cite{96} for $n$ up to $10^{13800}$; 
and \emph{Lehmer's totient problem} asking for composite numbers $n$ such that $\varphi(n) \mid (n-1)$, which is shown to have no solutions below $10^{22}$ or with less than 14 prime divisors~\cite{co}. 
Among equations with ``few'' known solutions, we can mention $\sum_{p \mid N}\frac{1}{p}-\frac{1}{N} \in \mathbb{N}$ with only $12$ known solutions called \emph{Giuga numbers} 
(sequence \texttt{A007850} in the OEIS~\cite{OEIS}) and $\sum_{p \mid N}\frac{1}{p}+\frac{1}{N} =1$ with only $8$ known solutions (sequence \texttt{A054377} in the OEIS~\cite{OEIS}) called \emph{primary pseudoperfect numbers}~\cite{BUT}.

In some cases, the search for new solutions to an equation only leads to the extension of the set of integers for which no solution is known. 
In other cases, theoretical and computational effort succeed in finding all the solutions. This is the case, for instance, for the equation $1^n + 2^n + \dotsb + n^n\equiv 19 \pmod{n}$ 
that we will show to have exactly $8$ solutions, namely $\{ 1,2,6,19,38,114,798,34314\}$.

For positive integers $k,n$, we define $S_k(n):=\sum_{i=1}^{n} i^k$. 
We will deal with congruences of the form 
\begin{equation}\label{eq:main}
S_n(n) \equiv a \pmod{n},
\end{equation} 
which is equivalent to $S_n(n-1) \equiv a \pmod{n}$. The following lemma shows that congruences \eqref{eq:main} is also equivalent to 
\begin{equation}\label{eq:bern}
n\cdot B_n \equiv a \pmod{n},
\end{equation} 
where $B_n$ is the $n$-th Bernoulli number.\footnote{The congruence $r_1\equiv r_2\pmod{n}$ for rational numbers $r_1$, $r_2$ is understood as $n$ divides the numerator of $r_1-r_2$.} 

\begin{lem}\label{lem:bern} For any positive integer $n$,
$$
S_n(n) \equiv n\cdot B_n \pmod{n}.
$$
\end{lem}
\begin{proof}
By Faulhaber's formula, we have
\begin{equation}\label{eq:faulhaber}
S_n(n) = \frac{1}{n+1} \sum_{j=0}^n (-1)^j\cdot \binom{n+1}{j}\cdot B_j\cdot n^{n+1-j},
\end{equation}
where by convention $B_1=-\frac{1}{2}$. In particular, for $j=1$, we have that the numerator of $B_j n^{n+1-j}=-\frac{n^n}{2}\equiv 0\pmod{n}$. For any odd $j\ne 1$, we have $B_j=0$, and thus the corresponding term in \eqref{eq:faulhaber} is zero as well.

Consider an even $j$. The Von Staudt--Clausen theorem implies the denominator of $B_j$ is square-free (in fact, it equals the product of all primes $p$ such that $(p-1)\mid j$) \cite{KES}. It follows that the denominator of $B_j\cdot n$ is coprime to $n$, and thus
$B_j\cdot n^{n+1-j}\equiv 0\pmod{n^{n-j}}$. Hence, $B_j\cdot n^{n+1-j} \equiv 0\pmod{n}$ for all $j<n$. 
Now reduction of \eqref{eq:faulhaber} modulo $n$ completes the proof.
\end{proof}

Let $\mathcal{M}_a$ denote the set of positive $n$ satisfying \eqref{eq:main} and \eqref{eq:bern} (Table~\ref{tab:MaOEIS}). 
From the Von Staudt--Clausen theorem, it is easy to see that $\mathcal{M}_0$ consists of the odd positive integers.
It is known~\cite{Sondow2011,GOS} that $\mathcal{M}_1=\{1,2,6,42,1806\}$.

\begin{table}[!t]
\label{tab:MaOEIS}
\caption{Values of $a$ and the sequence indices corresponding to $\mathcal{M}_a$ that are currently present in the OEIS~\cite{OEIS}.
The stars indicate when $\mathcal{M}_a$ is known to be finite.
Finiteness of $\mathcal{M}_p$ for primes $p\in\{2,3,7,19,43,79,193\}$ as well as for $p$ satisfying Theorem~\ref{th:COND} is established in the present work.}
\begin{center}
\begin{tabular}{|c||c|c|c|c|c|c|c|}
\hline
$a$ & $0^{\star}$ & $1^{\star}$ & $2^{\star}$ & $3^{\star}$ & $4$ & $5$ & $6$ \\ 
\hline
Index & {\tt A005408} & {\tt A014117} & {\tt A226960} & {\tt A226961} & {\tt A226962} & {\tt A226963} & {\tt A226964} \\ 
\hline
\hline
$a$ & $7^{\star}$ & $8$ & $9$ & $19^{\star}$ & $43^{\star}$ & $79^{\star}$ & $193^{\star}$ \\ 
\hline
Index & {\tt A226965} & {\tt A226966} & {\tt A226967} & {\tt A280041} & {\tt A280043} & {\tt A302343} & {\tt A302344} \\ 
\hline
\end{tabular}
\end{center}
\end{table}

In the present study, we focus on the case of $a$ being prime and address the problem of computing $\mathcal{M}_a$. 
We encounter both aforementioned situations:
in some cases, we are able to compute all the solutions to \eqref{eq:main} (and thus prove the finiteness of $\mathcal{M}_a$),
while 
in other cases, we find all solutions below certain large bounds (which are infeasible to reach by brute force). 

The main contribution of our work is the characterization of the solutions to the congruence \eqref{eq:main} and 
the development of an algorithm for computing the possible prime divisors of the solutions. 
Then, if the set of possible prime divisors is finite, the search for solutions can be restricted to products of these divisors and thus determine all the solutions.
Furthermore, we establish a connection of this problem to \emph{weak primary pseudoperfect numbers}, which enables computing all the solutions below $10^{30}$ with little computational effort.

\section{Characterization of $\mathcal{M}_p$}

The following lemma will be useful in the sequel.

\begin{lem}[\cite{GMO}] \label{LEM:GMO}
Let $d$, $k$, $n$, and $t$ be positive integers.
\begin{enumerate}[(i)]
\item If $d\mid n$, then
$$S_k(n)\equiv\frac{n}{d}\,S_k(d)\pmod{d}.$$
\item If $p>2$ is a prime, then
$$S_k(p^t)\equiv\begin{cases} 
-p^{t-1}\pmod{p^t}, & if\ p-1 \mid k;\\
0\qquad \pmod{p^t}, & otherwise.
\end{cases}$$
\item We have
$$S_k(2^t)\equiv\begin{cases} 
2^{t-1}\!\!\!\!\pmod{2^t},& \textrm{if\ $t=1$, or $t>1$ and $k>1$ is even};\\ 
-1\pmod{2^t}, & \textrm{if $t>1$ and $k=1$}; \\ 
0\ \ \pmod{2^t}, & \textrm{if $t>1$ and $k>1$ is odd}.\end{cases}$$
\end{enumerate}
\end{lem}

The following theorem gives a characterization of the set $\mathcal{M}_p$ in terms of the prime power factorization of its elements.

\begin{teor}\label{TEOR:P}
Let $p$ be a prime number. Then $n\in\mathcal{M}_p$ if and only if the following conditions hold:
\begin{enumerate}[(i)]
\item The prime power factorization of $n$ has form $n=p^sq_1\cdots q_r$, where $p,q_1,\dots, q_r$ are pairwise distinct primes and 
$0\leq s\leq 2$.
\item For every $i\in\{1,\dots,r\}$, $(q_i-1)\mid n$ and $\nicefrac{n}{q_i} + p \equiv 0\pmod{q_i}$.
\item If $s=1$, then $(p-1)\nmid n$.
\item If $s=2$, then $(p-1)\mid n$ and $\nicefrac{n}{p^2} + 1 \equiv 0\pmod{p}$.
\end{enumerate}
\end{teor}

\begin{proof}
We will work out the case of odd $p$; for $p=2$, the proof is similar. 

Let $n=2^tp^sq_1^{u_1}\cdots q_r^{u_r}$ be the prime power factorization of $n$. 
Then $S_n(n)\equiv p\pmod{n}$ if and only if 
$S_n(n)\equiv p\pmod{2^t}$, $S_n(n)\equiv p\pmod{p^s}$, and $S_n(n)\equiv p\pmod{q_i^{u_i}}$ for all $i\in\{1,\dots,r\}$.

By Lemma \ref{LEM:GMO}, $S_n(n)\equiv \dfrac{n}{2^t}S_n(2^t)\pmod{2^t}$, so $S_n(n)\equiv p\pmod{2^t}$ 
if and only if $t\leq 1$ with $\nicefrac{n}2+p\equiv 0\pmod{2}$ if $t=1$ by Lemma \ref{LEM:GMO}(iii).

Furthermore, by Lemma \ref{LEM:GMO}(i), $S_n(n)\equiv \dfrac{n}{p^s}S_n(p^s)\pmod{p^s}$, 
so $S_n(n)\equiv p\pmod{p^s}$ if and only if $\dfrac{n}{p^s}S_n(p^s)\equiv p\pmod{p^s}$, and we apply Lemma \ref{LEM:GMO}(ii) repeatedly. 
If $s=1$, the latter congruence holds if and only if $(p-1)\nmid n$. 
If $s>1$, it holds if and only if $(p-1)\mid n$ and $\nicefrac{n}{p^2}+1\equiv 0\pmod{p^{s-1}}$, with the latter congruence being possible only if $s\leq 2$.

Finally, by Lemma \ref{LEM:GMO}(i) again, $S_n(n)\equiv \dfrac{n}{q_i^{u_i}}S_n(q_i^{u_i})\pmod{q_i^{u_i}}$ 
and hence, since $p\neq q_i$, it follows from Lemma \ref{LEM:GMO}(iii) that $S_n(n)\equiv p\pmod{q_i^{u_i}}$ 
if and only if $(q_i-1)\mid n$ and $\nicefrac{n}{q_i}+p\equiv 0\pmod{q_i^{u_i}}$, with the latter congruence being possible only if $u_i\leq 1$.
\end{proof}

Theorem~\ref{TEOR:P} motivates us to consider a decomposition $\mathcal{M}_p=\mathcal{M}_p^{(0)}\cup\mathcal{M}_p^{(1)}\cup\mathcal{M}_p^{(2)}$, where
\begin{align*}
\mathcal{M}_p^{(0)}&=\{n\in\mathcal{M}_p\ :\ p\nmid n\},\\ 
\mathcal{M}_p^{(1)}&=\{n\in\mathcal{M}_p\ :\ p\mid\mid n\},\\
\mathcal{M}_p^{(2)}&=\{n\in\mathcal{M}_p\ :\ p^2\mid\mid n\}.
\end{align*}

We will now study each of these sets separately, using the following results.

\begin{lem}[\cite{GOS}]
\label{LEM:primes}
Let $\mathcal{P}$ be a non-empty set of primes $p$ such that
\begin{enumerate}[(i)]
\item $p-1$ is square-free; and
\item if $q$ is a prime divisor of $p-1$, then $q\in\mathcal{P}$.
\end{enumerate}
Then $\mathcal{P}$ is one of the sets $\{2\},\{2,3\},\{2,3,7\},$ or $\{2,3,7,43\}.$
\end{lem}

\begin{lem}[\cite{GOS}]
\label{LEM:N}
Let $\mathcal{N}$ be a set of positive integers $\nu$ such that
\begin{enumerate}[(i)]
\item $\nu$ is square-free, and
\item if $p$ is a prime divisor of $\nu$, then $p-1$ divides $\nu$.
\end{enumerate}
Then $\mathcal{N}\subseteq\{1,2,6,42,1806\}.$
\end{lem}

Lemma~\ref{LEM:N} implies the following result concerning $\mathcal{M}_p^{(0)}$.

\begin{lem}\label{PROP:M0}
Let $p$ be a prime. 
Then $\mathcal{M}_p^{(0)}\subseteq\{1,2,6,42,1806\}=\mathcal{M}_1$.
\end{lem}
\begin{proof}
Let $n\in\mathcal{M}_p^{(0)}$. Theorem~\ref{TEOR:P}(i) implies that $n$ is square-free. Moreover, Theorem~\ref{TEOR:P}(ii) implies that if $q$ is a prime divisor of $n$, then $q-1$ divides $n$. Hence, we can apply Lemma \ref{LEM:N} and the result follows.
\end{proof}

The following result is straightforward and completely determines the set $\mathcal{M}_p^{(0)}$.

\begin{lem}
Let $p$ be a prime. Then $\mathcal{M}_p^{(0)}=\{n\in\mathcal{M}_1:p\equiv 1\pmod{n}\}$.
\end{lem}

To study the set $\mathcal{M}_p^{(1)}$, we introduce the following set of primes associated with $p$.

\begin{defi}
For a prime $p$, we let $\mathcal{Q}_p$ be the set of prime numbers such that $q \in \mathcal{Q}_p$ if and only if the following conditions hold:
\begin{enumerate}[(i)]
\item $q-1$ is square-free;
\item $(p-1) \nmid (q-1)$;
\item if $t$ is a prime divisor of $q-1$,  then $t=p$ or $t\in \mathcal{Q}_p$.
\end{enumerate}
\end{defi}

In addition, we define the following set of integers composed of primes in $\mathcal{Q}_p$:
\begin{equation}\label{eq:Np}
\mathcal{N}_p:=\{n\in\mathbb{N}\ :\ \textrm{$n$ is square-free, $(p-1)\nmid n$, and for every prime $q\mid n$},\, q\in\mathcal{Q}_p\}.     
\end{equation}

\begin{cor}
Let $p$ be a prime. Then $\mathcal{M}_p^{(1)}\subseteq p\cdot\mathcal{N}_p$.
\end{cor}

\begin{proof}
Let $n\in\mathcal{M}_p^{(1)}$. Theorem \ref{TEOR:P} implies that $\nicefrac{n}{p}\in \mathcal{N}_p$ completing the proof.
\end{proof}

Finally, let us analyze the set $\mathcal{M}_p^{(2)}$. We will see that this set is empty in most cases. To do so, we first need the following lemma.

\begin{lem}\label{LEM:MP2}
Let $n\in\mathcal{M}_p^{(2)}$. If $q<p$ is a prime such that $q\mid n$, then $q\in\{2,3,7,43\}$.
\end{lem}

\begin{proof}
Let us consider the set of primes $\{q\ :\ \textrm{$q<p$ and $q\mid n$ for some $n\in\mathcal{M}_p^{(2)}$}\}$. 
Theorem~\ref{TEOR:P} implies that this set satisfies the conditions of Lemma \ref{LEM:primes}, completing the proof.
\end{proof}

\begin{cor}
For any prime $p\notin \{2,3,7,43\}$, the set $\mathcal{M}_p^{(2)}$ is empty.
\end{cor}

\begin{proof}
Assume that $n\in\mathcal{M}_p^{(2)}$. Since Theorem~\ref{TEOR:P} implies that $n=p^2q_1\cdots q_r$ and $(p-1)\mid n$, it follows that $p-1$ is square-free. Moreover, for the set of primes $S:=\{q\ :\ q\mid (p-1)\}$, Lemma \ref{LEM:MP2} implies that $S\subseteq\{2,3,7,43\}$. 
Thus, $p$ is a prime such that $p-1$ is square-free with prime divisors from the set $\{2,3,7,43\}$. 
It is easy to see that the only such primes are precisely $\{2,3,7,43\}$.
\end{proof}

The following result shows that in the remaining cases (i.e., for $p\in\{2,3,7,43\}$), the set $\mathcal{M}_p^{(2)}$ is also finite.

\begin{cor}\label{PROP:INC}
Let $p\in\{2,3,7,43\}$. Then $\mathcal{M}_p^{(2)}\subseteq p^2\cdot\mathcal{M}_1$.
\end{cor}

\begin{proof}
Define the set of primes $S:=\{q\ :\ q\ne p,\ q\mid n\ \textrm{for some $n\in\mathcal{M}_p^{(2)}$}\}$. 
Theorem~\ref{TEOR:P} implies that the set $S\cup\{p\}$ satisfies the conditions of Lemma~\ref{LEM:primes}, and hence $S\cup\{p\}\subseteq\{2,3,7,43\}$, i.e., $S\subsetneq\{2,3,7,43\}$. 
Now, the statement follows from the fact that every element in $\mathcal{M}_p^{(2)}$ is of the form $p^2q_1\cdots q_r$, where each $q_i\in S$.
\end{proof}

\begin{cor}\label{COR:DESC}
Let $p$ be a prime. Then
$$\mathcal{M}_p=\begin{cases}
\mathcal{M}_p^{(0)}\cup\mathcal{M}_p^{(1)}\subseteq\mathcal{M}_1\cup p\cdot \mathcal{N}_p, & \textrm{if $p\notin \{2,3,7,43\}$;}\\ 
\mathcal{M}_p^{(0)}\cup\mathcal{M}_p^{(1)}\cup\mathcal{M}_p^{(2)}\subseteq\mathcal{M}_1\cup p\cdot \mathcal{N}_p\cup p^2\cdot\mathcal{M}_1, & \textrm{otherwise.}
\end{cases}$$
In particular, if $\mathcal{N}_p$ is finite, then so is $\mathcal{M}_p$.
\end{cor}

\begin{cor}\label{PROP:743}
\begin{align*}
\mathcal{M}_7&=\{1,2,6,7,14,294,12642\},\\ 
\mathcal{M}_{43}&=\{1,2,6,42,43,86,258,77658\}.
\end{align*}
\end{cor}

\section{Algorithm for computing $\mathcal{Q}_p$ and $\mathcal{M}_p$}

Although Theorem~\ref{TEOR:P} gives a complete characterization of the set $\mathcal{M}_p$ for a prime $p$, 
from a practical point of view, Corollary~\ref{COR:DESC} is more useful for effective computation of this set. 
In particular, Corollary~\ref{COR:DESC} implies that in order to compute $\mathcal{M}_p$, it is enough to compute the set of primes $\mathcal{Q}_p$. 
Below we propose Algorithm~\ref{Alg1} that in the case of finite $\mathcal{Q}_p$ constructs it in a finite number of steps.
Namely, for an input prime $p$, Algorithm~\ref{Alg1} constructs a nested sequence of sets $X_1[p]\subseteq X_2[p]\subseteq \dots$. 
If this sequence stabilizes, the algorithm returns the limiting set denoted $\mathfrak{X}[p]$, which equals $\mathcal{Q}_p\cup\{p\}$ as we show in Theorem~\ref{th:algstop} below.

In Algorithm~\ref{Alg1}, $\PrimeParts(S)$ is defined as the set of primes in the set
$$\AllParts(S) := \{1+t\ :\ t=\prod_{q \in T}q\ \text{for some}\ T \subseteq S,\ (p-1)\nmid t\}.$$

\begin{algorithm}
\caption{Computing the set $\mathfrak{X}[p]$ for a given prime $p$.}\label{Alg1}
\begin{algorithmic}[1]
\State Let $X_1[p]:=\{2,p\}$
\For {$i=1,2,3,\dots$}
\State $X_{i+1}[p]:=X_i [p] \cup \PrimeParts(X_i[p])$
\If { $X_{i+1}[p] = X_i[p]$ }
\State \Return { $\mathfrak{X}[p]:=X_i[p]$ }
\EndIf
\EndFor
\end{algorithmic}
\end{algorithm}

\begin{teor}\label{th:algstop}
For every $i\geq 1$, we have that $X_i[p] \subseteq \mathcal{Q}_p\cup\{p\}$. 
Moreover, Algorithm~\ref{Alg1} stops if and only if $\mathcal{Q}_p$ is finite, in which case $\mathfrak{X}[p]=\mathcal{Q}_p\cup\{p\}$.
\end{teor}
\begin{proof} 
Let $\mathcal{Q'}_p=\mathcal{Q}_p\cup\{p\}$.
Clearly, $X_1[p] \subseteq \mathcal{Q'}_p$. 
Let us assume that there exists an index $i\geq 2$ such that $X_{i-1}[p] \subseteq \mathcal{Q'}_p$, but $X_i[p] \nsubseteq\mathcal{Q'}_p$. Consider the minimum element $q$ in $X_i[p] \setminus \mathcal{Q'}_p$. 
Since $q$ does not belong to $\mathcal{Q'}_p$, but $q-1$ is squarefree and $(p-1)\nmid (q-1)$, 
there exists a prime factor $q_1$ of $q-1$ that is not in $\mathcal{Q'}_p$ and thus not in $X_{i-1}[p]$ either. 
This contradicts the fact that every element of $X_i[p]\setminus\{2,p\}$ is of the form $1+p_1\cdots p_k$ with $p_j \in X_{i-1}[p]$. 
Hence, $X_i[p] \subseteq \mathcal{Q'}_p$ for every $i\geq 1$ as claimed.

Now, Algorithm~\ref{Alg1} constructs sets $X_1[p]\subseteq X_2[p]\subseteq \dots$, which are all subsets $\mathcal{Q'}_p$.
So, if $\mathcal{Q}_p$ is finite (and so is $\mathcal{Q'}_p$), the algorithm stops and returns the limiting set $\mathfrak{X}[p]$. 
Let us show that $\mathfrak{X}[p]=\mathcal{Q'}_p$. If $\mathfrak{X}[p]\subsetneq\mathcal{Q'}_p$, consider the minimum element $q$ in $\mathcal{Q'}_p \setminus \mathfrak{X}[p]$. 
Then $q-1=q_1\cdots q_r$ is squarefree, where $q_i \in  \mathcal{Q'}_p$. 
Since each $q_i<q$, the definition of $q$ implies that $q_i \in \mathfrak{X}[p]$. Hence, $q=1+q_1\cdots q_r \in \mathfrak{X}[p]$, since otherwise Algorithm~\ref{Alg1} would have not stopped. 
This contradicts the assumption of $\mathcal{Q'}_p \setminus \mathfrak{X}[p]$ being nonempty, and thus completes the proof.
\end{proof}

Once the set $\mathcal{Q}_p$ is obtained, one can easily compute $\mathcal{N}_p = \AllParts(\mathcal{Q}_p)$, and then use Corollary~\ref{COR:DESC} to find $\mathcal{M}_p$. Some interesting examples computed with Algorithm~\ref{Alg1} are given in the following table.

\begin{center}
\begin{footnotesize}
\begin{tabular}{|c|c|c|c|}
  \hline
  $p$ & stop & $\mathcal{Q}_p$ & $\mathcal{M}_p $ \\ \hline 
  $19$ & $i=8$ & $\{ 2 ,3,7,43,4903,168241543, 5773040306503\}$ & $\{ 1,2,6,19,38,114,798,34314\}$ \\ \hline 
  $79$ & $i=5$ & $\{2, 3, 7, 43, 3319, 1573207\}$ & $\{1, 2, 6, 79, 158, 474, 3318, 142674\}$ \\\hline 
  $193$ & $i=5$ & $\{2, 3, 7, 43, 348559\}$ & $\{1, 2, 6, 193, 386, 1158, 8106, 348558\}$ \\
  \hline
\end{tabular}
\end{footnotesize}
\end{center}

\

The following result establishes the finiteness of $\mathcal{Q}_p$ (and hence of $\mathcal{M}_p$) for a family of primes.

\begin{teor}\label{th:COND}
Let $A:=\{ 2, 6, 14, 42, 86, 258, 602, 1806 \}$ be the set of even divisors of $1806=2\cdot 3\cdot 7\cdot 43$. If a prime $p$ is such that the set  
$$\{1+\alpha p\ :\ \alpha\in A\}$$
does not contain any prime, then $\mathcal{M}_p \subseteq \mathcal{M}_1 \cup p \mathcal{M}_1$, and thus $\mathcal{M}_p$ is finite.
\end{teor}

\begin{proof} 
It is easy to see that primes $p\in \{2,3,7,43\}$ do not satisfy the condition as numbers $1+2\cdot 2$, $1+2\cdot 3$, $1+6\cdot 7$, and $1+1806\cdot 43$ are prime. On the other hand, for a prime $p\notin \{2,3,7,43\}$ satisfying the theorem condition, 
Algorithm~\ref{Alg1} returns $\mathfrak{X}[p] = \{2,3,7,43,p\}$, implying that $\mathcal{Q}_p=\{2,3,7,43\}$. Then the theorem statement follows from \eqref{eq:Np} and Corollary~\ref{COR:DESC}.
\end{proof}

There seem to exist many primes $p$ satisfying the condition of Theorem~\ref{th:COND} (sequence \texttt{A302345} in the OEIS~\cite{OEIS}). 
For example, the only such primes below $1000$ are
$$67, 97, 127, 163, 307, 317, 337, 349, 409, 521, 523, 547, 643, 709, 757, 811, 839, 857, 919, 967, 997.$$

We remark that there also exist primes $p$, for which $\mathcal{Q}_p$ and $\mathcal{M}_p$ are finite but do not satisfy the condition of Theorem~\ref{th:COND}. In particular, this holds for $p\in\{19, 79, 193\}$ present in the table above.

Unfortunately, in some cases we cannot determine if Algorithm~\ref{Alg1} stops due to the size of the involved sets of primes. 
For instance, for $p=5$, the set $X_5[p]$ contains 77 primes, and it seems infeasible to compute $X_6[p]$.

\section{Connection between $\mathcal{M}_p$ and primary pseudoperfect numbers}

When $\mathcal{Q}_p$ is infinite, Algorithm~\ref{Alg1} never stops. 
Nevertheless, there is an easy result that allows us to compute the elements of $\mathcal{M}_p$ below $p\cdot (8.49\times 10^{30})$ as explained below.

We recall that an integer $n\ge1$ is a \emph{weak primary pseudoperfect number}~\cite{GOS} if it satisfies the congruence:
$$\sum_{p\mid n}\frac{n}{p}+1\equiv 0\pmod{n}.$$
Let $\mathcal{W}$ be the set of all weak primary pseudoperfect numbers (sequence \texttt{A230311} in the OEIS~\cite{OEIS}).
The only known elements of $\mathcal{W}$ are
$$1,2,6,42,1806,47058, 2214502422, 52495396602, 8490421583559688410706771261086.$$
It is not even known if $\mathcal{W}$ is finite.

\begin{cor}\label{PROP:MpW}
Let $p$ be a prime. Then $\mathcal{M}_{p} \subseteq \mathcal{M}_1\cup p\cdot \mathcal{W}$.
\end{cor}

\begin{proof}
Let $n\in\mathcal{M}_p$. 
If $p\nmid n$, then $n\in\mathcal{M}_1$ by Lemma~\ref{PROP:M0}. 
On the other hand, if $p\mid n$, \cite[Corollary 1]{GOS} states that $\nicefrac{n}{p}\in\mathcal{W}$.
\end{proof}

Corollary~\ref{PROP:MpW} enables computing all the elements of $\mathcal{M}_p$ below $p\cdot \max\mathcal{W}$. 
It is enough to determine computationally if $S_n(n)\equiv p\pmod{n}$ for every element of
$\mathcal{M}_1\cup p\cdot \mathcal{W}$, which currently has up to $14$ known elements.

In some cases, it is possible to use \emph{ad hoc} arguments to prove that $\mathcal{M}_p$ is finite and, hence, to compute its elements. 
This is the case, e.g., for $p=2,3$. To see that both $\mathcal{M}_{2}$ and $\mathcal{M}_3$ are finite, we need to recall some ideas from \cite{GOS}. 
For every $Q\in\mathbb{N}$, we define 
$$\mathfrak{M}_Q:=\{n\in\mathbb{N}:S_{nQ}(nQ)\equiv n\pmod{nQ}\}.$$
If $\mathfrak{M}_Q\neq\emptyset$, then $Q\in\mathcal{W}$ (\cite[Corollary 1]{GOS}), and furthermore we have the following statement.

\begin{teor}[{\cite[Proposition 3]{GOS}}]
\label{PROP:NQ}
For a given weak primary pseudoperfect number $Q$, define the integer
$$\mathfrak{n}_Q:=\begin{cases}{\rm lcm}\left\{\frac{p-1}{\gcd(p-1,Q)}\ :\ \text{prime}\ p\mid Q\right\}, & \textrm{if $Q\neq 1$};\\ 1, & \textrm{if $Q=1$}.\end{cases}$$
Then $\mathfrak{M}_Q=\emptyset$ if and only if $(q-1)\mid \mathfrak{n}_QQ$ for some prime $q\mid \mathfrak{n}_Q$. 
Moreover, if $\mathfrak{M}_Q\neq\emptyset$, then $\mathfrak{n}_Q\mid n$ for every $n\in\mathfrak{M}_Q$ and, in particular, $\mathfrak{n}_Q=\min\mathfrak{M}_Q$.
\end{teor}

The following lemma is straightforward.

\begin{lem}\label{LEM:CAR}
Let $p$ be a prime. Then $n\in\mathcal{M}_p^{(1)}\cup\mathcal{M}_p^{(2)}$ if and only if $\nicefrac{n}{p}\in\mathcal{W}$ and $p\in\mathfrak{M}_{\nicefrac{n}p}$.
\end{lem}

While it seems to be plausible that the set $\mathcal{M}_p$ is finite for every prime $p$, 
there are many primes $p$ for which Algorithm~\ref{Alg1} fails to prove its finiteness. 
Nevertheless, in the previous setting, we can directly prove the finiteness of $\mathcal{M}_p$ for $p=2$ and $3$.

\begin{cor}\label{PROP:23}
If $p\in\{2,3\}$, then $\mathcal{M}_p$ is finite.
\end{cor}
\begin{proof}
By Lemma~\ref{PROP:M0} and Corollary~\ref{COR:DESC}, 
it is enough to show that $\mathcal{M}_p^{(1)}\cup\mathcal{M}_p^{(2)}$ is finite. 
Let us assume that $n\in \mathcal{M}_p^{(1)}\cup\mathcal{M}_p^{(2)}$, and observe that $\nicefrac{n}{p}\in\mathcal{W}$ and $p\in \mathfrak{M}_{\nicefrac{n}p}$ by Lemma~\ref{LEM:CAR}.

Let $p=2$ with $\nicefrac{n}2\in\mathcal{W}$ and $2\in\mathfrak{M}_{\nicefrac{n}2}$. 
Then Theorem~\ref{PROP:NQ} implies that $\mathfrak{n}_{\nicefrac{n}p}\mid 2$, i.e., $\mathfrak{n}_{\nicefrac{n}2}=1$ or $2$. 
Now, if $\mathfrak{n}_{\nicefrac{n}2}=2$, Theorem~\ref{PROP:NQ} implies that $\mathfrak{M}_{\nicefrac{n}2}=\emptyset$, a contradiction. 
Hence, $\mathfrak{n}_{\nicefrac{n}2}=1$, which implies that $(p-1)\mid \nicefrac{n}2$ for every $p\mid \nicefrac{n}2$, i.e., that $\nicefrac{n}2\in \mathcal{M}_1$ by Lemma~\ref{LEM:N}. 
Consequently, $\mathcal{M}_p^{(1)}\cup\mathcal{M}_p^{(2)}\subseteq 2\cdot\mathcal{M}_1$ is finite and so is $\mathcal{M}_2\subseteq \mathcal{M}_1\cup 2\cdot\mathcal{M}_1$.

Now, let $p=3$ with $\nicefrac{n}3\in\mathcal{W}$ and $3\in\mathfrak{M}_{\nicefrac{n}3}$. 
Again, we obtain that $\mathfrak{n}_{\nicefrac{n}3}=1$ or $3$. Since $n\in\mathcal{M}_p$ and $3\mid n$, if $n\neq 3$, Theorem~\ref{TEOR:P} implies that $(q-1)\mid n$ for every prime $q\mid \nicefrac{n}3$. 
In particular, $2\mid n$ and thus $2\mid \nicefrac{n}3$, so Theorem~\ref{PROP:NQ} implies that $\mathfrak{M}_{\nicefrac{n}3}=\emptyset$, which is a contradiction. 
Hence, $\mathfrak{n}_{\nicefrac{n}3}=1$ and $\mathcal{M}_3\subseteq \mathcal{M}_1\cup 3\cdot\mathcal{M}_1$ is finite.
\end{proof}

As a consequence, it is easy to compute the elements of $\mathcal{M}_p$ for $p=2,3$.

\begin{cor}\label{COR:23}
\begin{align*}
\mathcal{M}_2&=\{1,4,12,84,3612\},\\ \mathcal{M}_{3}&=\{1,2,3,18,126,5418\}.
\end{align*}
\end{cor}

In Lemma~\ref{PROP:743} and Corollary~\ref{COR:23}, we have established the finiteness and computed the elements of $\mathcal{M}_p$ for $p=2,3,7,43$. 
Recall that these are precisely the cases when $\mathcal{M}_p^{(2)}$ may be nonempty. 
In the remaining cases, $\mathcal{M}_p=\mathcal{M}_p^{(0)}\cup \mathcal{M}_p^{(1)}$. We will conclude this section with a characterization of $\mathcal{M}_p^{(1)}$ for $p\neq 2,3,7,43$.

\begin{lem}
Let $p\neq 2,3,7,43$ be a prime. Then $p\cdot \mathcal{M}_1=\{p,2p,6p,42p,1806p\}\subset\mathcal{M}_p^{(1)}$.
\end{lem}

\begin{proof}
The statement directly follows from Theorem~\ref{TEOR:P} and the definition of $\mathcal{M}_p^{(1)}$.
\end{proof}

\begin{cor}\label{PROP:NP}
Let $p\neq 2,3,7,43$ be a prime. Then $n\in\mathcal{M}_p^{(1)}$ if and only if $\nicefrac{n}p\in\mathcal{W}$, $\mathfrak{n}_{\nicefrac{n}p}\mid p$, and $(\mathfrak{n}_{\nicefrac{n}p}-1)\nmid \nicefrac{n}p$.
\end{cor}

\begin{proof}
Assume that $n\in\mathcal{M}_p^{(1)}$. By Lemma~\ref{LEM:CAR}, $\nicefrac{n}p\in\mathcal{W}$ and $p\in\mathfrak{M}_{\nicefrac{n}p}$. 
Hence, $\mathfrak{M}_{\nicefrac{n}p}\neq\emptyset$. By Theorem~\ref{PROP:NQ}, $\mathfrak{n}_{\nicefrac{n}p}\mid p$ and $(\mathfrak{n}_{\nicefrac{n}p}-1)\nmid \nicefrac{n}p$.

Conversely, assume that $\nicefrac{n}p\in\mathcal{W}$, $\mathfrak{n}_{\nicefrac{n}p}\mid p$, and $(\mathfrak{n}_{\nicefrac{n}p}-1)\nmid \nicefrac{n}p$. 
If $\mathfrak{n}_{\nicefrac{n}p}=1$, then similarly to the second part of the proof of Corollary \ref{PROP:23}, 
we obtain that $n\in p\cdot\mathcal{M}_1\subset\mathcal{M}_p^{(1)}$. 
On the other hand, if $\mathfrak{n}_{\nicefrac{n}p}=p$, then $\mathfrak{n}_{\nicefrac{n}p}-1=p-1\nmid \nicefrac{n}p$, and Theorem~\ref{PROP:NQ} implies that $p\in\mathfrak{M}_{\nicefrac{n}p}$. 
Then application of Lemma~\ref{LEM:CAR} completes the proof.
\end{proof}

Corollary~\ref{PROP:NP} enables computing (with little effort) all the elements of $\mathcal{M}_p$ below the product of $p$ and the largest known weak primary pseudoperfect number, 
which today gives the bound $p \cdot  8.49 \times 10^{30}$. 
It just remains to check if $S_n(n)\equiv p\pmod{n}$ for every element of $p \mathcal{W} \cup \mathcal{M}_{1}$. 
Implementing this idea, we obtain the following result.

\begin{cor} 
For every prime $p\neq 5$, we have that
$$ [1,p \cdot  8.49 \times 10^{30}] \cap \mathcal{M}_p \subseteq  \mathcal{M}_1   \cup p  \mathcal{M}_1.$$
\end{cor}

The prime $p=5$ is exceptional, since it is the only known prime $p$ for which there exist weak primary pseudoperfect numbers $Q$ satisfying $\mathfrak{n}_Q=p$. 
Namely, we have $\mathfrak{n}_{47058}=\mathfrak{n}_{2214502422}=5$. For prime $p=5$, we obtain the following result:

\begin{cor} 
$$\mathcal{M}_5 \cap [1,10^{31}] =\{1, 2, 5, 10, 30, 210, 9030, 235290, 11072512110\}.$$
\end{cor}

So, unless new weak primary pseudoperfect numbers are found, it is impossible to find more than 10 solutions to the congruence $S_n(n)\equiv p \pmod{n}$ with prime $p$. 
In other words, for a prime $p\neq 5$, finding a solution not from the set $\mathcal{M}_1 \cup p \mathcal{M}_1$ is equivalent to finding a new weak primary pseudoperfect number.

\section{Further work}

A natural extension of this work is, of course, to have a closer look at $\mathcal{M}_m$ with composite $m$. In this general case, we have the following analogue of Theorem~\ref{TEOR:P}.

\begin{teor}\label{TEOR:G}
Let $m=p_1^{r_1}\cdots p_s^{r_s}$ be an integer, where $p_1,\dots,p_s$ are pairwise distinct primes, $r_1,\dots,r_s$ are positive integers.
A positive integer $n$ belongs to $\mathcal{M}_m$ if and only if the following conditions hold:
\begin{enumerate}[(i)]
\item The prime power factorization of $n$ is given by $n=q_1\cdots q_r p_1^{t_1}\cdots p_s^{t_s}$, where 
$q_1,\dots,q_r$ are pairwise distinct primes not from $\{p_1,\dots,p_s\}$.
\item For every $j\in\{1,\dots,r\}$, $(q_j-1)\mid n$ and $\nicefrac{n}{q_j}+m\equiv 0\pmod{q_j}$.
\item For every $i\in\{1,\dots,s\}$, we have $t_i\in\{0,r_i,r_i+1\}$. Furthermore, if $t_i=r_i$, then $(p_i-1)\nmid n$; and if $t_i=r_i+1$, then $(p_i-1)\mid n$ and $\nicefrac{n}{p_i^{r_i+1}}+1\equiv 0\pmod{p_i}$.
\end{enumerate}
\end{teor}

\begin{proof}
Clearly, $n\in\mathcal{M}_m$ if and only if $S_n(n)\equiv m\pmod{q_j}$ for every $j\in\{1,\dots,r\}$ 
and $S_n(n)\equiv m\pmod{p_i^{t_i}}$ for very $i\in\{1,\dots,s\}$. 
It remains to apply Lemma~\ref{LEM:GMO} and argue just like in the proof of Theorem~\ref{TEOR:P}.
\end{proof}

Theorem~\ref{TEOR:G} enables construction of the set $\mathcal{M}_m$ for some particular values of $m$ as well as developing algorithms for computing 
the possible prime divisors of the elements of $\mathcal{M}_m$ (similarly to how we have done so in the prime case), but they are not operative. 
New ideas will have to be developed in order to attack this general situation. In any case, the following conjecture seems plausible. 

\begin{con} 
For every $m \in \mathbb{N}$ the set of solutions to the congruence $S_n(n)\equiv m \pmod{n}$ is finite.
\end{con}

\bibliographystyle{acm}
\bibliography{paper.bib} 

\end{document}